\documentclass{amsart}
\usepackage{amssymb}
\usepackage{amsmath}
\usepackage{amsthm}
\usepackage{amsfonts}
\usepackage{amsthm}

\newtheorem{prop}{Proposition}
\newtheorem{cor}{Corollary}
\newtheorem{lemma}{Lemma}
\newtheorem{thm}{Theorem}

\newcommand{\Cl}{\mathcal{C}l}
\newcommand{\calO}{\mathcal O}
\newcommand{\calP}{\mathcal P}
\newcommand{\Z}{\mathbb Z}
\newcommand{\Q}{\mathbb Q}
\newcommand{\N}{\mathbb N}
\newcommand{\p}{\frak p}
\newcommand{\q}{\frak q}
\newcommand{\leg}{\overwithdelims ()}

\begin{document}

\title{Non-unique factorization and principalization in number fields}
\date{\today}
\author{Kimball Martin}
\email{kmartin@math.ou.edu}
\address{Department of Mathematics, University of Oklahoma, Norman, OK 73011}


\begin{abstract}
Following what is basically Kummer's relatively neglected approach to non-unique factorization,
we determine the structure of the irreducible factorizations of an element $n$ in the ring of integers of a 
number field $K$.  Consequently, we give a combinatorial expression for the number of irreducible factorizations of $n$
in the ring.  When $K$ is quadratic, we show in certain cases how quadratic forms can be used to explicitly produce all irreducible
factorizations of $n$.
\end{abstract}

\maketitle

\section{Introduction}

One of the basic issues in algebraic number theory is the fact that for a number field $K$, and an integer $n \in \calO_K$, 
an irreducible factorization of $n$ may not be unique (up to ordering and units).  Historically, there were three major attempts to deal with this:
Gauss's theory of binary quadratic forms for quadratic fields $K$, Kummer's theory of ideal numbers, and Dedekind's theory of ideals.
Kummer's approach was largely abandoned in favor of Dedekind's powerful theory.  Ideals lead naturally to the general notion of the class group
$\Cl_K$, which provides a way of measuring the failure of unique factorization in $\calO_K$.   In fact there are precise ways
in which we can characterize the failure of unique factorization of $\calO_K$ in terms of the class group. 
For instance, Carlitz \cite{Carlitz} showed that
every irreducible factorization of $n$ over $\calO_K$
has the same length for all $n \in \calO_K$ if and only if the class number $h_K$ of $K$ is $1$ or $2$. 

Nevertheless, the precise way in which the class group determines the irreducible factorizations of $n$ is still not completely understood.  
Much of the
research on non-unique factorizations to date has been devoted to the study of lengths of factorizations, motivated by \cite{Carlitz}, 
determining which elements have unique factorization, and related questions
about asymptotic behavior.  For an introduction to this subject, see  \cite{Nark}, \cite{HKsurvey} or \cite{GHKsurvey}.  See
\cite{GHK} for a more comprehensive reference.

In this note we will give an explicit description of the structure of the irreducible factorizations of $n$ in $\calO_K$ in terms of the prime ideal
factorization of $(n)$ which depends only upon the class group $\Cl_K$.   Thus we have a very precise description of how the class group 
measures the failure of unique factorization in $\calO_K$.  For example, Carlitz's result is an immediate corollary.  To do this, we use the idea
of principalization, which is essentially Kummer's theory of ideal numbers, framed in the language of Dedekind's ideals.
Specifically, we pass to an extension $L$ which principalizes $K$, i.e., every ideal of $K$ becomes principal in
$L$.  
This means that all irreducible factorizations of $n \in \calO_K$ come from different groupings of a single factorization in $\calO_L$.  

To understand the irreducible factorizations of $n$ in a more quantitative way, one natural question is, what is the number $\eta_K(n)$ of 
(nonassociate) irreducible factorizations of $n$ in $\calO_K$?  Of course if $h_K=1$, then $\eta_K(n) = 1$ for all $n \in \calO_K$.  
If $h_K=2$, then Chapman, Herr and Rooney \cite{CHR} established a formula for $\eta_K(n)$ in terms of the
prime ideal factorization of $(n)$ in $\calO_K$.  However, it is a rather
complicated recursive formula on the number of prime ideal factors of $(n)$.  (In fact the work \cite{CHR} treats the more general case of Krull
domains.)   For work on determining for which $n$ satisfy $\eta_K(n)=1$ see \cite{Nark}; for asymptotic results on $\eta_K(n)$ see \cite{HKacta}, \cite{HKark} or \cite{GHK}.

We obtain, for an arbitrary number field $K$ and $n \in \calO_K$, a relatively simple  combinatorial 
expression for $\eta_K(n)$, which appears to be about as simple as one could hope for.
This formula is particularly simple in the case $h_K=2$, and we begin in Section \ref{sec2}
by explaining how to treat the standard class number $2$ example of $K=\Q(\sqrt{-5})$.  The expression we get for $\eta_K(n)$ is valid for 
any $K$ with class number $2$, and is considerably nicer than the formula in \cite{CHR}.
One can in fact treat the case of $K=\Q(\sqrt{-5})$ with quadratic forms, and this is what we do in Section \ref{sec2}.  
This yields, more than just the structure of the factorizations of $n$ in $\calO_K$, the explicit irreducible factorizations
of $n$ in terms of the representations of the primes dividing $n$ (say if $n \in \Z$) by certain quadratic forms.

In Section \ref{sec3}, we treat the case of an arbitrary number field $K$, and discuss some simple consequences.   We remark that these
results should in fact apply to more general Krull domains than $\calO_K$ by the theory of type monoids \cite[Section 3.5]{GHK}, but this is
not our focus here.
In Section \ref{sec4}, we revisit the approach using quadratic forms presented in Section \ref{sec2} for quadratic fields $K$.

We would like to thank Daniel Katz for helpful comments, and suggesting the use of formal power series.
We also appreciate the comments provided by the referee after a careful reading of the manuscript.

\section{An example: class number $2$} \label{sec2}

Let $K=\Q(\sqrt{-5})$.  This field has discriminant $\Delta = -20$ and class group $\Cl_K\simeq \Z/2\Z$.  
Denote by $\frak C_1$ the set of principal ideals in $\calO_K$ and $\frak C_2$ the set of nonprinicpal ideals of $\calO_K$.
Now the reduced (positive binary quadratic) forms of discriminant $\Delta$ are $Q_1(x,y) = x^2+5y^2$ and $Q_2(x,y) = 2x^2+2xy+3y^2$.

Let $\calP_0$ denote the primes $p \in \N$ which are not represented by $Q_1$ or $Q_2$ and $\calP_i$ denote the primes $p \in \N$
which are represented by $Q_i$ for $i=1, 2$.  
Then $\calP_0$ is the set of inert primes in $K/\Q$, $\calP_1$ is the set of primes $p$
such that the ideal $p\calO_K$ factors into two principal ideals in $\calO_K$, and $\calP_2$ is the set of primes $p$ such that $p\calO_K$
factors into two nonprincipal ideals of $\calO_K$.  

Set 
\begin{align*} \calP_i^{ram} &= \{ p \in \calP_i : p \text{ is ramified in }K \},  \text{ and} \\
\calP_i^{unr} &= \{ p \in \calP_i : p \text{ is unramified in }K\}.
\end{align*} 
Explicitly, we have $\calP_0 = \{ p :  p \equiv 11, 13, 17, 19 \mod 20\}$, $\calP_1^{ram}
= \{ 5 \}$, $\calP_1^{unr} = \{ p: p \equiv 1, 9 \mod 20 \}$, $\calP_2^{ram} = \{2 \}$ and $\calP_2^{unr} = \{ p : p \equiv 3, 7 \mod 20 \}$.

If $p \in \calP_0 \cup \calP_1$ then any prime ideal $\p$ of $\calO_K$ lying above $p$ is in $\frak C_1$, and if $q \in \calP_2$, then any 
prime ideal of $\calO_K$ lying above $q$ is in $\frak C_2$.  Specifically, if $q = 2\in \calP_2^{ram}$, then $q\calO_K = \frak r^2$ where
$\frak r$ is the prime ideal $(2,1+\sqrt{-5})$ of $\calO_K$, and if 
$q  \in \calP_2^{unr}$ then $q = \q \bar \q$ where $\q$ and $\bar \q$ are distinct prime ideals of $\calO_K$.  Here $\bar \q$ 
denotes the conjugate ideal of $\q$ in $K/\Q$.

Now let $n \in \calO_K$ be a nonzero nonunit, and write the prime ideal factorization of $(n)$ as 
\[ (n) = \p^{d_1}_1 \cdots \p^{d_r}_r \q^{e_{11}}_1 \bar \q_1^{e_{12}} \cdots \q^{e_{1s}}_s \bar \q_s^{e_{2s}} \frak r^f, \]
where each $\p_i \in \frak C_1$, $\q_j \in \frak C_2$ with conjugate $\bar \q_j$, and 
the $\p_i$'s, $\q_j$'s, $\bar \q_j$'s and $\frak r$ are all distinct.  
Since each $\p_i = (\pi_i)$ for some irreducible $\pi_i$ of $\calO_K$, any irreducible factorization
of $n$ must contain (up to units) $\pi_1^{d_1} \cdots \pi_r^{d_r}$.  
Thus it suffices to consider irreducible factorizations of $n'=n/(\pi_1^{d_1} \cdots \pi_r^{d_r})$.

Let $q_j$ be the prime in $\N$ such that $\q_j$ lies above $q_j$.  Since $\q_j$ is nonprincipal, we must have that $q_j \in \calP_2$, i.e., $q_j$
is represented by $Q_2$.  Note that we can factor the quadratic form into linear factors
\begin{equation} \label{223fact}
Q_2(x,y) = (\sqrt 2x + \frac{\sqrt 2 + \sqrt{-10}}2 y) (\sqrt 2x + \frac{\sqrt 2 - \sqrt{-10}}2 y)
\end{equation}
over the field $L=K(\sqrt{2})$.  Hence, while $q_j$ is irreducible over $\calO_K$ (otherwise the prime ideal factors of $q_j\calO_K$ would be
principal), the fact that $q_j = Q_2(x,y)$ for some $x,y$ gives us a factorization $q_j= \alpha_{1j} \alpha_{2j}$ in $L$ where
$\alpha_{1j} = \sqrt 2x + \frac{\sqrt 2 + \sqrt{-10}}2 y$ and $\alpha_{2j} = \sqrt 2x + \frac{\sqrt 2 - \sqrt{-10}}2 y$.  
Since $\sqrt{2}, \frac{\sqrt 2 \pm \sqrt{-10}}2 \in \calO_L$, we have $\alpha_{ij} \in \calO_L$ (in fact irreducible).  

Observing that $\alpha_{1j}$ and $\alpha_{2j}$ are conjugate with respect to the nontrivial element of $\mathrm{Gal}(K/\Q)$,
the ideals $(\alpha_{1j}) \cap \calO_K$
and $(\alpha_{2j}) \cap \calO_K$ must be conjugate ideals of $\calO_K$ which divide $q_j$, and hence in some order equal 
$\q_j$ and $\bar \q_j$.  Thus, up to a possible relabeling, we can write
$\q_j\calO_L = (\alpha_{1j})$ and $\bar \q_j \calO_L = (\alpha_{2j})$.  Similarly $\frak r \calO_L = (\sqrt 2)$.  To simplify notation below,
we set $\alpha_{00} = \sqrt 2$ and $e_{00} = f$.

This means the following.  If $n' = \prod \beta_k$ is any irreducible factorization of $n'$ in $\calO_K$, we have
$(n') = \q^{e_{11}}_1 \bar \q_1^{e_{12}} \cdots \q^{e_{1s}}_s \bar \q_s^{e_{2s}} \frak r^{e_{00}} = \prod (\beta_k)$ as ideals of $\calO_K$.  
By unique factorization of prime ideals, any $(\beta_k) =  \q^{g_{11}} \bar \q_1^{g_{12}} \cdots \q^{g_{1s}}_s \bar \q_s^{g_{2s}} \frak r^{g_{00}}$
where $0 \leq g_{ij} \leq e_{ij}$.  Passing to ideals of $\calO_L$, we see
each $\beta_k$ is a product of the $\alpha_{ij}$.  In other words, all the irreducible factorizations of $n'$ in $\calO_K$, up to units,
come from different groupings of the factorization $n' = u\prod \alpha_{ij}^{e_{ij}}$ in $\calO_L$, where $u$ is a unit of $\calO_K$.

Thus to determine the factorizations of of $n'$ in $\calO_K$, it suffices to determine when a product of the $\alpha_{ij}$ is an irreducible
element of $\calO_K$.  But this is simple! Note from the factorization of $Q_2(x,y)$ in (\ref{223fact}), we see that each $\alpha_{ij} \in \sqrt 2 K$.
Hence the product of any two $\alpha_{ij}$ lies in $K$, and therefore $\calO_K$, and must be irreducible since no individual
 $\alpha_{ij}  \in \calO_K$.  What we have done, together with the fact that the $\alpha_{ij}$'s are all nonassociate (they generate different ideals), proves the following.
 
If $\{ a_i \}$ is a collection of distinct objects, denote the multiset containing each $a_i$ with cardinality $m_i$ by $\{ a_i^{(m_i)}\}$.

\begin{prop} \label{prop1}
 Let $K=\Q(\sqrt{-5})$, $L=K(\sqrt 2)$ and $n \in \calO_K$ be a nonzero nonunit.  Write the prime ideal factorization of $(n)$ in $\calO_K$ as
 $(n) = \prod \p_i^{d_i} \prod \q_j^{e_j}$,
where each $\p_i \in \frak C_1$, $\q_j \in \frak C_2$ and the $\p_i$'s and $\q_j$'s are all distinct.  Let $\pi_i \in \calO_K$ and 
$\alpha_j \in \calO_L$ such that $\p_i = (\pi_i)$ and $\q_j\calO_L = (\alpha_j)$.  
Then the irreducible nonassociate factorizations of $n$ are precisely $n = u \prod \pi_i^{d_i} \prod \beta_k$ where $u$ is a unit, each $\beta_k$ is a product of two (not necessarily distinct) $\alpha_j$'s
and $\prod \beta_k = \prod \alpha_j^{e_j}$.

In particular $\eta_K(n)$ is the number of ways we can arrange the multiset $\{ \alpha_j^{(e_j)} \}$ in pairs, i.e., the number of partitions of this 
multiset into sub-multisets of size $2$.  In other words, if the number of distinct $\q_j$'s
is $m$, then $\eta_K(n)$ is the coefficient of $\prod x_j^{e_j}$ in the formal power series expansion of $\prod_{i\leq j} \frac 1{1-x_i x_j}$ in
$\Z[[x_1, x_2, \ldots, x_m]]$.
\end{prop}

Note the final sentence is essentially a tautology.

Thus, in addition to $\eta_K(n)$, we have provided an explicit determination of the irreducible factorizations of an arbitrary $n \in \calO_K$, provided
one knows the (irreducible) factorization in $\calO_L$.  (The point is that there is a choice of which of $\alpha_{1j}$ and $\alpha_{2j}$ is 
$\q_j\calO_L$ and which is $\bar \q_j \calO_L$ in the above argument.)
However if $n \in \Z$, then it suffices to know the prime factorization 
$n = 2^f \prod p_k'^{d_k'} \prod p_i^{d_i}  \prod q_j^{e_j}$ in $\Z$, where each $p_k' \in \calP_0$, $p_i \in \calP_1$, $q_j \in \calP_2^{unr}$
and they are all distinct.  For then each $p_k'$ is irreducible in $\calO_K$, the irreducible factorization each of $p_i$ in $\calO_K$ 
is given by solving $p_i = x^2+5y^2=(x+\sqrt{-5} y)(x-\sqrt{-5}y)$, and the factorization $q_j = \alpha_{1j}\alpha_{2j}$ in $\calO_L$ above
is given by solving $q_j = 2x^2+2xy+3y^2$.  Here there is no need to worry which of $\alpha_{1j}$ and $\alpha_{2j}$ correspond to which of
$\q_j$ and $\bar \q_j$ since both $\q_j$ and $\bar \q_j$ both occur to the same exponent $e_j$.

\medskip
Except for the explicit factorization we get from the quadratic forms $Q_1$ and $Q_2$ above, all of this is true for arbitrary number fields with class number $2$.  

\begin{prop} \label{prop2} 
If $K$ is a number field with class number $2$, and $n \in \calO_K$ is a nonzero nonunit, write the prime ideal factorization of $(n)$ as
$(n) = \prod \p_i^{d_i} \prod_{j=1}^m \q_j^{e_j}$ where each $\p_i$ is principal, $\q_j$ is nonprincipal, and the $\p_i$'s and $\q_j$'s are all distinct.
Consider the formal power series $f(x_1, \ldots, x_m) \in \Z[[x_1, \ldots, x_m]]$ formally given by $f(x_1, \ldots, x_m) = \prod_{i,j} \frac 1{1-x_i x_j}$.
Then $\eta_K(n)$ is the number of ways we can arrange the multiset $\{ x_j^{(e_j)} \}$ in pairs, i.e., the coefficient of $\prod x_j^{e_j}$
in $f(x_1, \ldots x_m)$.
\end{prop}

One can either conclude this result from the work of \cite{CHR} together with our example of $K=\Q(\sqrt{-5})$ or remark it is a special
case of our main result, Theorem \ref{thethm}, below.
In \cite{CHR}, the authors prove a recursive formula for $\eta_K(n)$, being recursive on the number of 
nonprincipal prime ideal factors of $(n)$, which is independent of field $K$ (in fact it is valid more generally for Krull domains also, but we will not 
stress this).  As the formula in \cite{CHR} is rather complicated, we will not state their complete formula here, but just give the first
two cases to give the reader an idea of form of their expressions.  
In the notation of the corollary above, they show, assuming $e_1 \leq e_2 \leq \cdots \leq e_m$,
that  $\eta_K(n) = \eta_{\mathbb X_{m+1}}(e_1,e_2, \ldots, e_m, \frac{e_1 + \ldots + e_m}2)$ where 
\begin{align*}
 \eta_{\mathbb X_2}(x_1,x_2) &= \lfloor \frac{\min(x_1,x_2)}2 \rfloor + 1 \\
 \eta_{\mathbb X_3}(x_1,x_2,x_3) & = \sum_{j=0}^{\lfloor \frac{x_1}2 \rfloor} \sum_{k=0}^{x_1-2j} \eta_{\mathbb X_2}(x_2-k,x_3-x_1+2j+k),
\end{align*} 
and the expression for $\eta_{\mathbb X_{m+1}}$ involves an $m$-fold summation over $\eta_{\mathbb X_m}$.

Hence our approach of principalization and factoring the form $Q_2(x,y)$ in $K(\sqrt{2})$ provides a much nicer combinatorial answer to the
question of what is $\eta_K(n)$.   
 We now proceed to see what our result says in some simple cases.
For the rest of this section, we maintain the notation of Proposition \ref{prop2}.

The first thing we observe is that $\sum e_i$ must be even.

\begin{cor} An irreducible factorization of $n$ is unique, i.e., $\eta_K(n)=1$, if and only if (i) there is at most one nonprincipal prime ideal dividing
$(n)$, or (ii) $(n) = \prod \p_i^{d_i} \cdot \q_1^e \q_2$ where the $\p_i$'s are principal, the $\q_i$'s are nonprincipal, and $e$ is odd.
\end{cor}

In particular, to return to the original example of $K=\Q(\sqrt{-5})$, if $n \in \Z$, then $\eta_K(n) = 1$ if and only if (i) it is not divisible by any primes
in $\calP_2^{unr}$, i.e., any primes of the form $q \equiv 3, 7 \mod 20$, or (ii) $n = q \prod p_i^{d_i}$ where the $p_i$'s and $q$ are primes
with $q \equiv 3, 7 \mod 20$ and each $p_i \in \calP_0 \cup \calP_1$, i.e., $p_i \not \equiv 3, 7, \mod 20$ and $p_i$ odd.
This classification of $n \in \N$ with $\eta_{\Q(\sqrt{-5})}(n)=1$ was previously established by Fogels  
using an approach similar in spirit to ours in \cite{Fogels}, where he used this to show that ``almost all''
$n \in \N$ do not have unique factorization.

\begin{cor} If $m=2$, then $\eta_K(n) = \lfloor \frac{\min(e_1,e_2)}2 \rfloor + 1$.
\end{cor}

Note this matches with the formula for $\eta_{\mathbb X_2}(x_1,x_2)$ in \cite{CHR}.  This was observed earlier in the case of elementary abelian
$2$- class groups  (\cite[Example 1]{HKark}).

\begin{proof} We want to count the number of ways we can pair $e_1$ $x_1$'s and $e_2$ $x_2$'s.  This is simply determined by the
number of $x_1$'s which are paired up with $x_2$'s.  This can be any number $k$ between $0$ and $\min(e_1,e_2)$ such that $e_i - k$ is
even.
\end{proof}

In the special case $K=\Q(\sqrt{-5})$, this means if $q \in \calP_2^{unr}$ then $h_K(q^e) = \lfloor \frac{e}2 \rfloor + 1$.

\begin{cor} \label{cor3}
 If $e_1=e_2= \cdots = e_m = 1$, then $\eta_K(n) = (m-1)!!$.
\end{cor}

\begin{proof} This is just the number of ways in which we can arrange the set $\{ x_1, \ldots, x_m \}$ in pairs, which is
$(m-1)!! = (m-1)(m-3)\cdots 1$.
\end{proof}

When $K=\Q(\sqrt{-5})$ and $q_1, \ldots, q_k$ are distinct primes in $\calP_2^{unr}$, this means $h_K(q_1 \cdots q_k) = (2k-1)!!$.

\section{General Number Fields} \label{sec3}

Let $K$ be an arbitrary number field and let $\Cl_K = \{ \frak C_i \}$ be the ideal class group of $K$.  
Denote the class of principal ideals in $\calO_K$ by $\frak I$.

We say $K_i$ is a {\em
principalization field} for $\frak C_i$ if $K_i$ is an extension of $K$ such that every ideal in $\frak C_i$ becomes principal in $\calO_{K_i}$.
Such a field always exists.  For example if $\frak C_i$ has order $m$, then for any ideal $\frak a \in \frak C_i$, we have $\frak a^m$ is principal.
Say $\frak a^m = (a)$.
Consequently $\frak a$, and therefore every ideal in $\frak C_i$, becomes principal in the field $K_i = K(\sqrt[m] a)$.

We say $L$ is a {\em prinicpalization field} for $K$ if every ideal in $\calO_K$ becomes principal in $\calO_L$.  For instance if $K_i$ is a
principalization field for $\frak C_i$ for each $\frak C_i \in \Cl_K$, then the compositum $L= \prod K_i$ is a principalization field for $K$.
By the principal ideal theorem of class field theory, the Hilbert class field of $K$ is a principalization field for $K$.

If $\alpha, \beta \in \calO_K$ and $\alpha = u \beta$ for a unit $u \in \calO_K$, i.e., if $\alpha$ and $\beta$ are associates, write
$\alpha \sim \beta$.

\begin{thm} \label{thethm}
Let $K$ be a number field and $\Cl_K= \{ \frak C_i \}$.  Let $n \in \calO_K$ be a nonzero nonunit.  
Suppose the prime ideal factorization of $n\calO_K$ is $(n) = \prod_{(i,j) \in T} \p_{ij}$
  where the $\p_{ij}$'s are (not necessarily distinct) prime ideals such that 
 $\p_{ij} \in \frak C_i$, and $T$ is some finite index set.  Let $K_i$ be a principalization field for $\frak C_i$,
  so $\p_{ij} \calO_{K_i} = (\alpha_{ij})$ for some $\alpha_{ij} \in \calO_{K_i}$.   Let $L = \prod K_i$.
  
  Then the irreducible factorizations of $n$ in $\calO_K$ are precisely the factorizations of the form $n = \prod \beta_l$ where 
  $\prod \beta_l \sim \prod \alpha_{ij}$ in $\calO_L$ and each $\beta_l$ is of the form
   $\beta_l \sim \prod_{(i,j) \in S} \alpha_{i j}$ in $\calO_L$ for $S$ a minimal (nonempty) subset of $T$ such that
   $\prod_{(i,j) \in S} \frak C_i = \frak I$.  (Here each $\beta_l$ is irreducible in $\calO_K$.)
\end{thm}   

In other words, all irreducible factorizations $n$ in $\calO_K$ come from different groupings of the factorization $n \sim \prod \alpha_{ij}$
 in $\calO_L$.  Now a grouping of terms of this factorization in $\calO_L$ gives an irreducible factorization
in $\calO_K$ if and only if every group of terms gives an irreducible element of $\calO_K$ (possibly up to a unit in $\calO_L$).  (We will
call such a grouping {\em irreducible}.)
A product of $\alpha_{ij}$'s gives an element of $\calO_K$ if and only if the corresponding product of ideal classes $\frak C_i$ is trivial in
$\Cl_K$, and this element of $\calO_K$ will be irreducible if and only if no proper subproduct of the corresponding ideal classes is trivial.

It should be clear that this theorem gives a precise way that the class group measures the failure of unique factorization in $\calO_K$.  In particular,
the larger the class group, the more complicated the structure of the irreducible factorizations of an element can become.  Simple explicit
examples are given at the end of Section \ref{sec4}.  This theorem also connects Kummer's and Dedekind's approaches to resolving 
non-unique factorization in $\calO_K$.

We also remark that one could take each $K_i = L$ for any principalization field $L$ of $K$, but we will see in the next section reasons
why one may not always want to do this.  In fact, for specific $n$, $L$ need not be a principalization field for $K$, but just for the ideal classes
containing ideals dividing $n\calO_K$.

\begin{proof}  Suppose $n=\prod \beta_l$ is an irreducible factorization of $n$ in $\calO_K$, i.e., each $\beta_l$ is a (nonunit) irreducible.  
By unique factorization of prime ideals,
each $(\beta_l)$ is a subproduct of $\prod \p_{ij}$.  
Write $(\beta_l) = \prod_{(i,j) \in S} \p_{ij}$ where $S \subseteq T$.
Since $(\beta_l)$ is principal, the subproduct of prime ideals yielding
$(\beta_l)$ must be trivial in the class group, i.e., $\prod_{(i,j) \in S} \frak C_i = \frak I$.  
Further, $S$ must be minimal such that the corresponding product in the class group is trivial,
otherwise we would be able to write $(\beta_l)$ as a product of two principal ideals, contradicting irreducibility.

Write $S=\{ (i_1,j_1), (i_2, j_2), \ldots, (i_r, j_r) \}$, so that
\[ (\beta_l) = \p_{i_1j_1} \p_{i_2j_2} \cdots \p_{i_rj_r} \]
Observe 
\[ \beta_l \calO_{K_{i_1}} = (\alpha_{i_1 j_1}) \frak P_{i_2j_2} \cdots \frak P_{i_rj_r}, \]
 where
$\frak P_{i_ij_i} = \p_{i_ij_i} \calO_{K_{i_1}}$.  Passing to $\calO_{K_{i_1}K_{i_2}}$ and using the fact that $\p_{i_2j_2} = (\alpha_{i_2j_2})$
in $\calO_{K_{i_2}}$, we see that
\[ \beta_l \calO_{K_{i_1}K_{i_2}} = (\alpha_{i_1 j_1}) (\alpha_{i_2j_2}) \frak P^{(2)}_{i_3j_3} \cdots \frak P^{(2)}_{i_rj_r},\]  
where
$\frak P^{(2)}_{i_ij_i} = \frak P_{i_ij_i} \calO_{K_{i_1}K_{i_2}}$.  Proceeding inductively, we obtain
\[ (\beta_l) = (\alpha_{i_1 j_1}) (\alpha_{i_2j_2}) \cdots (\alpha_{i_rj_r}) \]
as ideals in $\calO_L$, yielding (ii) as asserted in the theorem.  

This proves that any irreducible factorization of $n$ in $\calO_K$ is of the form stated above, namely that any irreducible factorization of $n$
is obtained from a grouping of the terms in the (not necessarily irreducible) 
factorization $n = u \prod \pi_i \prod \alpha_{ij}$ in $\calO_L$ such that each group of terms
is minimal so that the corresponding product in the class group $\Cl_K$ is trivial.  (Here $u$ is some unit.)  It remains to show that any such
grouping gives an irreducible factorization of $\calO_K$. 

It suffices to show that if $S$ is a minimal subset of $T$ such that  $\prod_{(i,j) \in S}\frak C_i = \frak I$, then $u\prod_{(i,j)\in S} \alpha_{ij}$ is an 
irreducible element of $\calO_K$ for some unit $u \in \calO_L$.   
Suppose $S$ is such a subset.  Then
$\prod_{(i,j) \in S} \q_{ij} = (\beta)$ for some $\beta \in \calO_K$.  As above, looking at ideals in $\calO_L$, we see
$\beta \sim \prod_{(i,j) \in S} \alpha_{ij}$, hence the product on the right is, up to a unit of $\calO_L$, an element of $\calO_K$.  
If $\beta$ were reducible, say $\beta = \gamma \gamma'$ where $\gamma, \gamma' \in \calO_K$ are nonunits, then by unique factorization
into prime ideals, we would have $(\gamma) = \prod_{(i,j) \in S'} \q_{ij}$ where $S'$ is a proper subset of $S$, i.e., 
$\prod_{(i,j) \in S'} \frak C_i = \frak I$, contradicting the minimality of $S$.
\end{proof}

\begin{cor}  Let $K$ be a number field and $\Cl_K = \{ \frak C_i \}$.  Let $n \in \calO_K$ be a nonzero nonunit.  Suppose
$(n) = \prod_{(i,j) \in T} \p_{ij}^{e_{ij}}$, where the $\p_{ij}$'s are distinct prime ideals, each $\p_{ij} \in \frak C_i$, and $T$ is some index set.
Let $U$ be the multiset $U = \{ (i,j)^{(e_{ij})} : (i,j) \in T \}$. Then
 $\eta_K(n)$ is the coefficient of $\prod x_{ij}^{e_{ij}}$ in the formal power series 
   \[ f(x_{ij}) = \prod_S \frac 1{1-\prod_{(i,j) \in S}x_{ij}} \in \Z[[x_{ij}]] \, , \hspace{10pt} (i,j) \in T, \]
   where $S$ runs over all minimal sub-multisets of $U$ such that the product $\prod_{(i,j) \in S} \frak C_i =\frak I$.
    Combinatorially, $\eta_K(n)$ is the number of ways one can partition the multiset $\{ x_{ij}^{e_{ij}} \} $ into minimal subsets $V$
   such that $\prod_{x_{ij} \in V} \frak C_i = \frak I$.
\end{cor}

\begin{proof} Let $K_i$ be a principalization field for $\frak C_i$, and write $\p_{ij}\calO_{K_i} = (\alpha_{ij})$.  Set $L = \prod K_i$.
Then we have $n \sim \prod_T \alpha_{ij}^{e_{ij}} = \prod_U \alpha_{ij}$ over $\calO_L$.  By the theorem,
the irreducible factorizations of $n$ in $\calO_K$ correspond to the partitions of $U$ into minimal sub-multisets $S$ such that
$\prod_S \frak C_i = \frak I$.  Hence it remains to show that any two distinct partitions give nonassociate factorizations of $n$.  

It suffices to prove that if $\prod_S \alpha_{ij} \sim \prod_{S'} \alpha_{ij}$ over $\calO_L$ for two sub-multisets $S, S'$ of $T$, then $S = S'$.
But this hypothesis means that 
\[ \prod_S \p_{ij} \calO_L = \prod_S \alpha_{ij} \calO_L = \prod_{S'} \alpha_{ij} \calO_L = \prod_{S'} \p_{ij} \calO_L. \]
Intersecting our ideals with $\calO_K$ gives $\prod_S \p_{ij} = \prod_{S'} \p_{ij}$, which means $S=S'$ by unique factorization into 
prime ideals.
\end{proof}

The current approach to investigating lengths and number of factorizations has primarily been through block and type monoids
\cite[Chapter 3]{GHK}.  Our theorem essentially gives the theory of block and type monoids in the case of rings of integers of number fields.
In particular it can be used to provide new proofs of many known results in the theory of non-unique factorizations.  Here we just illustrate
the most basic example of Carlitz's result.

If $n$ is a nonzero nonunit in $\calO_K$ and $n = \prod \alpha_i$ where each $\alpha_i$ (not necessarily distinct) is a (nonunit) irreducible
of $\calO_K$, we say the number of $\alpha_i$'s occurring in this product (with multiplicity) is the {\em length} of this factorization.

\begin{cor} {\em(\cite{Carlitz})} Let $K$ be a number field.  Every irreducible factorization of $n$ in $\calO_K$ has the same length for
all nonzero nonunits $n \in \calO_K$ if and only if $h_K \leq 2$.
\end{cor}

\begin{proof} It is immediate from the theorem (or Proposition \ref{prop2}) that if $h_K \leq 2$, then every irreducible factorization of an element
must have the same length.  Suppose $h_K > 2$.  

First suppose $\Cl_K$ has an element $\frak C$ of order $e > 2$.  Then let $\p \in \frak C$ and $\q \in \frak C^{-1}$ be prime ideals of $\calO_K$.
Let $n \in \calO_K$ such that $(n) = \p^e \q^e$.  Then one irreducible factorization of $n$ is corresponds
to the grouping $(n) = (\p\q)(\p\q) \cdots (\p\q)$ which has length $e >2$.  Another irreducible factorization of $n$ corresponds to the grouping
$(n) = (\p^e )(\q^e)$, which has length $2$.

Otherwise $\Cl_K$ has at least three elements $\frak C_1$, $\frak C_2$ and $\frak C_3 = \frak C_1 \frak C_2$ of order $2$.  
Let $\p_i \in \frak C_i$ be a prime
ideal of $\calO_K$ for each $i=1,2,3$.   Let $n \in \calO_K$ such that $(n) = \p_1^2 \p_2^2 \p_3^2$.  The two different groupings
$(\p_1 \p_2 \p_3)(\p_1 \p_2 \p_3)$ and $(\p_1^2)(\p_2^2)(\p_3^2)$ give irreducible factorizations of $n$ of lengths $2$ and $3$.
\end{proof}

This proof might be considered a slight simplification, but it does not differ in any essential way from Carlitz's original proof.
However, looking at this proof suggests that if $\Cl_K \simeq \Z/h\Z$ then the ratio of the maximal length of an irreducible factorization of $n$
to a minimal length is bounded by $\frac h2$ for any nonzero nonunit $n \in \calO_K$.  In fact this is true, and the maximum value of this ratio
is called the {\em elasticity} $\rho_K$ of  $K$.  More generally, the {\em Davenport constant} $D(\Cl_K)$ of $\Cl_K$ is defined to be
the maximal $m$ such that there is a product of length $m$ which is trivial in $\Cl_K$ but no proper subproduct is.  Then the above theorem
can be used to provide a new proof the known result  (e.g., see \cite{Nark}) that $\rho_K = D(\Cl_K)/2$.

\medskip
Specializing to certain cases, we can obtain simple formulas for $\eta_K(n)$ or criteria on when $\eta_K(n) = 1$.  A few examples were
given in the case of class number $2$ in the previous section.  Here we give two more simple examples for arbitrary class number.

\begin{cor} Let $K$ be a quadratic field and $p \in \Z$ a rational prime.  Let $\p$ be a prime of $\calO_K$ above
$p$, and let $m$ be the order of $\p$ in $\Cl_K$.  If $m=1$ or $p$ is ramified in $K/\Q$ then $\eta_K(p^n) = 1$ for all $n \in \N$.
Otherwise, $\eta_K(p^n) = \lfloor \frac nm \rfloor + 1$.
\end{cor}

\begin{proof}  If $m=1$, the statement is obvious.  If $p$ is ramified, then $p\calO_K = \p^2$, and again the result is immediate from our main result.
Otherwise $p\calO_K = \p \bar \p$ where $\bar \p \neq \p$ and $\bar \p$ is the inverse of $\p$ in $\Cl_K$.  Then any irreducible of $\calO_K$
dividing $p$ corresponds to one of the groupings $\p \bar \p$,  $\p^m$ or  $\bar \p^m$.  The number of times $\p^m$ appears in an
irreducible grouping of $\p^n \bar \p^n$ is the same as the number of times $\bar \p^m$ will appear.
Hence the irreducible factorizations of $p^n$ in $\calO_K$ are determined by the number of $\p^m$'s which appear in an
irreducible groupings of $\p^n \bar \p^n$.
\end{proof}

We remark that in \cite{HKark}, Halter-Koch showed for any number field $K$ and $x \in \calO_K$ (or more generally a Krull monoid), 
$\eta_K(x^n) = An^d + O(n^{d-1})$ for some $A \in \Q$ and $d \in \Z$.

\begin{cor} Let $K$ be a number field and $\frak C \in \Cl_K$ be an ideal class of order $m$.  Suppose $n \in \calO_K$ such that
$(n) = \p_1 \p_2 \cdots \p_k$ where the $\p_i$'s are distinct prime ideals in $\frak C$.  Then $\eta_K(n)$ is the number of partitions of
$\{ 1, 2, \ldots, k \}$ into subsets of size $m$, i.e., $\eta_K(n) = \frac{k!}{(m!)^{k/m}(k/m)!}$.
\end{cor}

This is immediate from our main result, and a generalization of Corollary \ref{cor3}.

\section{Explicit factorizations in quadratic fields} \label{sec4}

As we pointed out earlier, the approach via quadratic forms in Section \ref{sec2} in some sense gives the irreducible factorizations of an element
of $\calO_K$ in a more explicit fashion.  Specifically, one does not know {\em a priori} the elements $\alpha_{ij}$ occurring in Theorem \ref{thethm}
explicitly.  Therefore one might ask in what generality can one apply the prinicipalization argument from Section \ref{sec2} using quadratic forms.
First we must restrict to the case of quadratic fields.

From now on, unless otherwise stated, let $\Delta$ be a fundamental discriminant and $K=\Q(\sqrt{\Delta})$ be the quadratic field of
discriminant $\Delta$.
Suppose $Q(x,y) = ax^2+bxy+cy^2$ is a primitive quadratic form of discriminant $\Delta$.   Then $Q(x,y)$ factors into
linear factors
\begin{equation} \label{qfact}
Q(x,y) = ax^2+bxy+cy^2 = \left( \sqrt a x + \frac{b+\sqrt \Delta}{2\sqrt a} y \right) \left( \sqrt a x + \frac{b-\sqrt \Delta}{2\sqrt a} y \right)
\end{equation}
over $K'=K(\sqrt{a})$.  Clearly $\sqrt a \in \calO_{K'}$.  On the other hand $\beta^{\pm} = \frac{b\pm \sqrt \Delta}{2\sqrt a} \in \calO_{K'}$ if and only if
the norm $N_{K'/K}(\beta^\pm)$ and trace $Tr_{K'/K}(\beta^\pm)$ lie in $\calO_K$.  

Note that $K'=K$ if and only if $a = m^2$ or $a=m^2\Delta$ for some
$m \in \Z$.  The latter is not possible since $Q$ is primitive.  The former implies that $\beta^\pm \in \calO_{K'} = \calO_K$ if and only if $a=1$.  

Suppose $K' \neq K$ and write $\mathrm{Gal}(K'/K) = \{ 1, \sigma \}$.  Then $\sigma(\beta^\pm) = -\beta^\pm$ so 
we always have $Tr_{K'/K}(\beta^\pm) = 0 \in \calO_K$.  On the other hand $N(\beta^\pm) = \frac{b(b\pm \sqrt{\Delta})}{2a}-c$, which lies in 
$\calO_K$ if and only if $b|a$.  If $b|a$, then $Q(x,y)$ is called {\em ambiguous}.  Hence we have shown

\begin{lemma} \label{lemma1}
Let $Q(x,y)=ax^2+bxy+cy^2$ be a primitive form of discriminant $\Delta$.  Then $Q$ factors into integral linear forms in 
$\Q(\sqrt{\Delta}, \sqrt a)$ if and only if $Q$ is ambiguous. 
\end{lemma}

In other words, we can use the factorization of a quadratic form to principalize the corresponding ideal class if and only if the quadratic form
is ambiguous.  This makes sense because an ideal class corresponds to an ambiguous form if and only if it has order $\leq 2$ in the class group.
On the other hand, the linear factorization of a binary quadratic form always happens over a quadratic extension, but one needs to use
an extension of degree $m$ to principalize an ideal class of order $m$ in $\Cl_K$.  

To see this last assertion, suppose $\frak a$ is an ideal of order $m$ 
in $\Cl_K$, so that $\frak a^m = (\alpha)$.  If $L$ principalizes $\frak a$, say $\frak a \calO_L = (\beta)$, then $\beta^m \calO_L = \alpha \calO_L$.
Hence $\sqrt[m]{u\alpha} \in \calO_L$ for some unit $u \in \calO_K$.  No $k$-th root of $u\alpha$ is contained in $K$ for $1\neq k | m$ since
$\frak a$ has order $m$.  Therefore $m | [L:K]$.

\medskip
We now set up our notation for the statement and proof of the main result of this section.
Let $\frak I$ be the class of principal ideals in $\Cl_K$, and 
$\frak C_1, \ldots, \frak C_k$ be the ideal classes in $\Cl_K$ of order $2$.   We assume $k \geq 1$.

If $Q(x,y) = ax^2+bxy+cy^2$ is primitive of discriminant $\Delta$,
 we define the ideal in $\calO_K$ corresponding $Q$ to be $(a, \frac{b-\sqrt{\Delta}}2 )$.  
We will say two forms are (weakly) equivalent if their corresponding ideals are equivalent, so that the equivalence classes of forms form a group
isomorphic to $\Cl_K$.  It is easy to see that $Q$ and $-Q$ correspond to the same ideal if $Q$ is ambiguous.

Let $Q_j(x,y) = a_j x^2 + b_jxy + c_jy^2$ be an ambiguous form corresponding to an ideal in $\frak C_j$.  
Set $K_j = K(\sqrt{a_j})$ and $L = K_1 K_2 \cdots K_k$.   

\begin{lemma} \label{lemma3}
   $K_j$ is a principalization field for $\frak C_j$.  Hence
  $L$ is a principalization field for $\frak C_1, \ldots, \frak C_k$.
\end{lemma}

\begin{proof}
  Let $\frak a$ be the ideal of $\calO_K$ corresponding to $Q_j$, and $\bar{\frak a}$ be its conjugate.  One easily checks that
  $\bar{\frak a} = \frak a$ and ${\frak a}^2 = (a_j)$.  Thus $\frak a \calO_{K_j} = (\sqrt{a_j})^2$.  Since $\frak a \in \frak C_j$,
  $K_j$ principalizes any ideal in $\frak C_j$.
\end{proof}

Though we do not need this for the proposition below, it would be decent of us to determine the structure of $L/K$.  This
follows from the following.

\begin{lemma} \label{lemma2}
Let $Q(x,y) = ax^2+bxy+cy^2$ and $R(x,y) = dx^2 + exy + fy^2$ be primitive ambiguous forms of discriminant $\Delta$.  
Let $\frak a = (a, \frac{b-\sqrt{\Delta}}2 )$ and $\frak b = (d, \frac{e-\sqrt{\Delta}}2)$
be the ideals of $\calO_K$ corresponding to $Q$ and $R$.
Then $K(\sqrt a) = K(\sqrt d)$ implies $\frak a$ and $\frak b$ are equivalent.
\end{lemma}

\begin{proof} Write $b=ra$ and $e=sd$.  Note that $b^2-4ac = \Delta$ then implies $a | \Delta$.  Since $\Delta$ is either squarefree or $4$
times a squarefree number, we have that $a$ is either squarefree, $2$ times a squarefree number or $4$ times a squarefree number.  On the
other hand, if $4 | a$, then $16 | b^2-4ac = \Delta$, which is not possible.  Hence $a$ is squarefree.  Similarly $d$ is a squarefree divisor of 
$\Delta$.

For any squarefree $n | \Delta$ and $m \in \Z$, we have $\sqrt m \in K(\sqrt n)$ if and only if $m = a, \frac{\Delta}{a}$,  $\Delta,$
$\frac{ \Delta}{4a}$ or $\frac \Delta 4$.  Set $\Delta' = \frac \Delta 4$ if $4 | \Delta$ and $\Delta' = \Delta$ otherwise.
Thus $K(\sqrt a) = K(\sqrt d)$ if and only if $d=a$ or $d=\frac{\Delta'}a$.

First suppose $d=a$.  
Note that dividing $r^2a^2-4ac = s^2a^2-4af$ by $a$ implies $r^2a \equiv s^2a \mod 4$, which implies $r \equiv s \mod 2$ since $4 \nmid a$.
But this means the ideals $\frak a = (a, \frac{ra-\sqrt{\Delta}}2 )$ and $\frak b = (a, \frac{sa-\sqrt{\Delta}}2 )$ are in fact equal.

If $d = \frac{\Delta'}a$, we may replace $R(x,y)$ with the equivalent form $R(y,-x)$, thus interchanging $d$ and $f$, and negating $e$.  
This means both $e$ and $f$ are now divisible by $\frac{\Delta'} a$, so $d$ cannot be by primitivity.  This means $d$ must be $\pm a$, which
we have just dealt with.  (If $d=-a$, we can replace $R$ by $-R$, which corresponds to the same ideal.) 
\end{proof}

We remark that this lemma gives the following well known result.

\begin{cor}  If $\Cl_K$ contains a subgroup isomorphic to $(\Z/2\Z)^r$, then $\Delta$
has at least $r+1$ distinct prime divisors.
\end{cor}

\begin{proof} Since there must be at least $2^r$ pairwise equivalent ambiguous forms $a_i x^2+b_ixy + c_i y^2$ of discriminant $\Delta$,
the above lemma and its proof imply that the $a_i$'s and $\frac \Delta{a_i}$'s are distinct divisors of $\Delta$.  Each $a_i$ is always squarefree,
and if $\frac \Delta{a_i}$ is not squarefree, then $\frac \Delta{4a_i}$ is, and it is distinct from the other divisors.  Thus $\Delta$ has at least
$2^{r+1}$ distinct squarefree divisors, so it must have at least $r+1$ distinct prime factors.
\end{proof}

One could refine this had we been using the notion of proper equivalence classes of quadratic forms, which we do not need for our purpose.
Precisely, if $r$ is maximal so that $\Cl_K$ contains a subgroup isomorphic to $(\Z/2\Z)^r$, then one can show that there are 
either $r+1$ or $r+2$ distinct prime divisors of $\Delta$.  The first case occurs when the extended genus field of $K$ equals the genus field of $K$, 
and the second when they are different.  (See, e.g., \cite{Janusz}).

However, our interest in the previous lemma is in the structure of $L/K$ (which is closely related to the genus field and extended 
genus field of $K$, but different from both in general).  We know $\{ \frak I, \frak C_1, \ldots, \frak C_k \}$ is the subgroup of 
$\Cl_K$ generated by all elements of order $2$.  We put $r$ such that $2^r=k+1$ so that this subgroup is isomorphic to 
$(\Z/2\Z)^r$.

\begin{cor} $L$ is an abelian extension of $K$ of degree $2^r$ and {\em Gal}$(L/K) \simeq (\Z/2\Z)^r$.
\end{cor}

\begin{proof} Clearly $[L:K] \leq 2^r$ and is a power of $2$ by construction.  Moreover $L/K$ is Galois and
the Galois group is an elementary abelian $2$-group because $L$ is obtained from $K$ by adjoining square roots of $K$.
By the previous lemma, $L/K$ has $2^r - 1$ subextensions of
degree $2$ over $K$, so $[L:K]=2^r$.  
\end{proof}

\begin{prop} Let $n = \prod p_i^{d_i} \prod q_{jk}^{e_{jk}} \prod r_\ell^{f_\ell} \in \N$ where the $p_i$'s are primes in $\N$ represented by the principal form $Q_0(x,y)=x^2+b_0xy+c_0y^2$ of discriminant $\Delta$, the $q_{jk}$'s are primes in $\N$
represented by $Q_j$ and the $r_\ell$'s are primes in $\N$ not represented by any form of discriminant $\Delta$.  Write each $p_i = Q_0(u_i,v_i)$
and $q_{jk} = Q_j(x_{jk},y_{jk})$ for $u_i, v_i, x_{jk}, y_{jk} \in \Z$.  Let
\[ \alpha_i^\pm = u_i + \frac{b_0\pm \sqrt \Delta}{2} v_i \]
and
\[ \beta_{jk}^\pm = \sqrt a_j x_{jk} + \frac{b_j\pm \sqrt \Delta}{2\sqrt {a_j}} y_{jk}. \]
Then the irreducible factorizations of $n$ in $\calO_K$, up to units, are precisely given by the $\calO_K$-irreducible groupings of the factorization
\[ n = \prod_i (\alpha^+_i \alpha^-_i)^{d_i} \prod_{jk} (\beta^+_{jk} \beta^-_{jk})^{e_{jk}} \prod r_\ell^{f_\ell} \]
in $\calO_L$.
\end{prop}

By an $\calO_K$-{\em irreducible grouping} of a product $\prod \gamma$ in $\calO_L$, we of course mean a grouping of the terms such that
the product of each group of terms is (up to a unit of $\calO_L$) an irreducible in $\calO_K$.  In the above proposition, each $\alpha_i^\pm$ and
$r_\ell$ is already an irreducible of $\calO_K$, and the elements $\beta_{jk}^\pm$ correspond to the ideal class $\frak C_j$.  A product of these
$\beta_{jk}^\pm$'s is, up to a unit of $\calO_L$, an irreducible in $\calO_K$ if and only if the corresponding product of ideal classes is trivial but 
no proper subproduct is.  In fact, such a product of $\beta_{jk}^\pm$'s must actually be an irreducible of $\calO_K$, since the fact that
$\beta_{jk}^\pm \in \sqrt{a_j} K$ implies such a product lies in $\calO_K$.

\begin{proof} It is obvious that any prime $p_i$ represented by $Q_0$ satisfies $p_i\calO_K = \p_1 \p_2$ for some principal prime 
ideals $\p_1$ and $\p_2$ of $\calO_K$, since $Q_0$ factors over $K$.  
Further any prime $q_{jk}$ represented by $Q_j$ satisfies $q_{jk} \calO_K = \q_1 \q_2$ for some prime ideals
$\q_1, \q_2 \in \frak C_j$ (see \cite[p. 143]{BS}).  Lastly each $r_\ell$ is inert in $K/\Q$.  Now apply Theorem \ref{thethm}.
\end{proof}

The above gives a complete answer for the factorization of rational integers $n$ in $\calO_K$ when $\Cl_K \simeq (\Z/2\Z)^r$, i.e., when there is
one class per genus in the form class group, and a partial answer for other quadratic fields.  We end with two examples and some remarks
on principalization fields.

\medskip
\noindent
{\bf Example 1.} Let $\Delta = -87$.  Then $K=\Q(\sqrt{-87})$ has class number $h_K = 6$.  The principal form is $Q_0(x,y) = x^2+xy+22y^2$
and there is one other ambiguous form up to equivalence, $Q_1(x,y)=3x^2+3xy+8y^2$.  
Let $n=14145=3 \cdot 5 \cdot 23 \cdot 41$.  We see $3 = Q_1(1,0)$,
${-87 \leg 5} = -1$ so $5$ is not represented by a form of discriminant $\Delta$, $23=Q_0(1,1)$ and $41 = Q_1(1,2)$.  Let 
\[ \alpha^{\pm} = 1  + \frac{1\pm \sqrt {-87}}{2} = \frac{3 \pm \sqrt{-87}}2, \]
\[ \beta_1 = \sqrt 3, \]
\[ \beta_2^\pm = \sqrt 3  + 2 \frac{3\pm \sqrt {-87}}{2\sqrt {3}}  =  2\sqrt{3} \pm \sqrt{-29} \]
Then the irreducible factorizations of $n$ in $\calO_K$ are given by the $\calO_K$-irreducible groupings of the factorization
\[ n= \alpha^+ \alpha^{-} \beta_1^2 \beta_2^+ \beta_2^- \cdot 5 \]
in $\calO_L$, where $L=K(\sqrt{3})$.
Specifically, $\eta_K(n) = 2$ and the factorizations are
\[ (\alpha^+) (\alpha^{-}) (\beta_1^2) (\beta_2^+ \beta_2^-)  5 =  \frac{3 + \sqrt{-87}}2 \cdot \frac{3 - \sqrt{-87}}2 \cdot 3 \cdot 41 \cdot 5, \]
\[ (\alpha^+) (\alpha^{-}) (\beta_1 \beta_2^+) (\beta_1 \beta_2^-)  5 =  \frac{3 + \sqrt{-87}}2 \cdot \frac{3 - \sqrt{-87}}2 
(6 + \sqrt{-87})(6-\sqrt{87}) \cdot 5. \]

\medskip
\noindent
{\bf Example 2.} Let $\Delta = -21$.  Then $K=\Q(\sqrt{-21})$ has class group $\Cl_K \simeq (\Z/2\Z)^2$.  We take for our ambiguous forms
the principal form $Q_0(x,y) = x^2+21y^2$, $Q_1(x,y) = 2x^2+2xy+11y^2$, $Q_2(x,y) = 3x^2+7y^2$ and $Q_3(x,y) = 14x^2+14xy+5y^2$.
(Note that all of these are reduced, except for $Q_3$ which is equivalent to the reduced form $5x^2+4xy+5y^2$.) 
 
A prime  $p \in \N$ is represented by $Q_0$ if $p \equiv 1, 25, 37 \mod 84$, by $Q_1$ if $p=2$ or $p \equiv 11, 23, 71
\mod 84$, by $Q_2$ if $p=3,7$ or $p \equiv 19, 31, 55 \mod 84$, and by $Q_3$ if $p \equiv 5, 17, 41 \mod 84$.

Let $n = 46189 = 11\cdot 13 \cdot 17\cdot 19$, so $2 = Q_1(1,0)$, $11=Q_1(0,1)$, $13$ is not represented by a form of discriminant $\Delta$,
$17 = Q_3(1,-3)$ and $19=Q_2(2,1)$.  Set 
\[ \beta_{1}^\pm = \frac{2\pm 2\sqrt{-21}}{2\sqrt 2} = \frac{1\pm \sqrt{-21}}{\sqrt 2} \]
\[ \beta_2^{\pm} = 2\sqrt{3} \pm {\sqrt{-7}} \]
\[ \beta_3^{\pm} = \sqrt{14} - 3 \cdot \frac{14 \pm 2\sqrt{-21}}{2\sqrt{14}} = \frac{-\sqrt{7} \pm 3\sqrt{-3}}{\sqrt 2}. \]
Then the irreducible factorizations of $n$ in $\calO_K$ are given by the $\calO_K$-irreducible groupings of the factorization
\[ n = \beta_{1}^+ \beta_{1}^- \beta_{2}^+ \beta_2^- \beta_3^+ \beta_3^- \cdot 13\]
in $\calO_L$ where $L = K(\sqrt 2, \sqrt 3, \sqrt{14})$.
Precisely, there are $\eta_K(n) = 5$ of them and they are
$ 13 (\beta_{1}^+ \beta_{1}^-)(\beta_{2}^+ \beta_2^-)( \beta_3^+ \beta_3^-) = 13 \cdot 11 \cdot 19 \cdot 17$,
$ 13 (\beta_{1}^+ \beta_2^+ \beta_3^+)(\beta_1^- \beta_2^- \beta_3^-)$, 
$ 13 (\beta_{1}^+ \beta_2^+ \beta_3^-)(\beta_1^- \beta_2^- \beta_3^+)$,
$13 (\beta_{1}^+ \beta_2^- \beta_3^+)(\beta_1^- \beta_2^+ \beta_3^-)$, and
$13 (\beta_{1}^+ \beta_2^- \beta_3^-)(\beta_1^- \beta_2^+ \beta_3^+)$.

\medskip
\noindent
{\bf Final remarks.}
In the case $K$ is a quadratic field with class group $\Cl_K \simeq (\Z/2\Z)^r$, we have constructed a principalization field $L$ which
is Galois over $K$ and $\mathrm{Gal}(L/K) \simeq \Cl_K$.  Further, $L$ is unramified outside of any primes dividing $2\Delta$.  In fact, by using
 $K_j=K(\sqrt{-a_j})$ instead of $K(\sqrt{a_j})$ when $a_j \equiv 3 \mod 4$, we can ensure $L=\prod K_j$ is unramified outside of any (finite) primes dividing 
 $\Delta$.   Moreover, this is not equal to the Hilbert class field $H$ of $K$ in general, as our earlier example of $K=\Q(\sqrt{5})$ shows.
 (It is of course closely related to $H$, and more generally to the genus field of $K$.)

In general for a number field $K$ it is natural to ask, what we can say about the minimal abelian extensions $L$ which principalize $K$?
By the remarks after Lemma \ref{lemma1}, we know $m | [L:K]$ for every cyclic group of order $m$ contained in $\Cl_K$.
One might be tempted to posit that $[L:K] \geq h_K$, or even that $\mathrm{Gal}(L/K)$ contains $\Cl_K$, but this turns out to be false.  For instance, the Hilbert class field $H$ of $K$ is an
abelian extension of $K$ with $\mathrm{Gal}(L/K) \simeq \Cl_K$ and always 
principalizes $K$, but proper subextensions of $H$ may also principalize $K$ (\cite{HS}, \cite{Iwasawa1}, \cite
{Iwasawa2}).  We will not survey the literature on principalization, but refer to the expositions \cite{Jaulent}, \cite{Miyake} and \cite{Suzuki},
as well as point out the recent works \cite{Gras} and \cite{Bosca} which study extensions of $K$ not contained in its Hilbert class field.

\end{document}